\documentclass{amsart}
\usepackage[pagebackref]{hyperref}
\usepackage{amscd,amssymb}
\include{diagrams}

\newcommand{\Q}{\mathbb Q}
\newcommand{\N}{\mathbb N}
\newcommand{\Z}{\mathbb Z}
\newcommand{\R}{\mathbb R}

\newcommand{\fa}{{\mathfrak a}}
\newcommand{\fp}{{\mathfrak p}}
\newcommand{\eps}{\varepsilon}
\newcommand{\cok}{\text{cok}\,}
\newcommand{\Gal}{\operatorname{Gal}}
\newcommand{\too}{\longrightarrow}
\newcommand{\Cl}{\text{Cl}}
\newcommand{\disc}{\text{disc }}
\newcommand{\OO}{\mathcal O}

\newtheorem{thm}{Theorem}
\newtheorem{prop}{Proposition}
\newtheorem{lem}{Lemma}
\newtheorem{cor}{Corollary}

\title{Ideal class groups of cyclotomic number fields I}
\author{Franz Lemmermeyer} 			
\address{M\"orikeweg 1 \\ 73480 Jagstzell \\ Germany}
\email{hb3@ix.urz.uni-heidelberg.de}
\subjclass{Primary 11 R 18; Secondary 11 R 29}

\begin{document}

\begin{abstract}
Following Hasse's example, various authors have been deriving 
divisibility properties of minus class numbers of cyclotomic 
fields by carefully examining the analytic class number formula. 
In this paper we will show how to generalize these results to
CM-fields by using class field theory. Although we will only need 
some special cases, we have also decided to include a few results 
on Hasse's unit index for CM-fields as well, because it seems
that our proofs are more direct than those in Hasse's book \cite{H}.
\end{abstract}

\maketitle

\section{Notation}  

Let $K \subset L$ be number fields; we will use the following notation:
\begin{itemize}
\item $\OO_K$ is the ring of integers of $K$;
\item $E_K$ is its group of units;
\item $W_K$ is the group of roots of unity contained in $K$;
\item $w^{}_K$ is the order of $W^{}_K$;
\item $\Cl(K)$ is the ideal class group of $K$;
\item $[\fa]$ is the ideal class generated by the ideal $\fa$;
\item $K^1$ denotes the Hilbert class field of $K$, that is the 
         maximal abelian extension of $K$ that is unramified at 
         all places;
\item $j^{}_{K \to L}$ denotes the transfer of ideal classes for number 
        fields $K \subset L$, i.e. the homomorphism $\Cl(K) \to \Cl(L)$ 
        induced by mapping an ideal $\mathfrak a$ to
         $\mathfrak a \mathfrak O_L$;
\item $\kappa_{L/K}$ denotes the capitulation kernel 
        $\ker \, j^{}_{K \to L}$;
\end{itemize}
Now let $K$ be a CM-field, i.e. a totally complex quadratic extension 
of a totally real number field; the following definitions are
standard:

\begin{itemize}
\item $\sigma$ is complex conjugation;
\item $K^+$ denotes the maximal real subfield of $K$; this is
      the subfield fixed by $\sigma$.
\item $\Cl^-(K)$ is the kernel of the map 
      $N_{K/K^+}: \Cl(K) \to \Cl(K^+)$ and is called the minus class 
      group; 
\item $h^-(K)$ is the order of $\Cl^-(K)$, the minus class number;
\item $Q(K)$ $=(E_K:W^{}_KE_{K^+}) \in \{ 1, 2 \}$ is Hasse's unit index.
\end{itemize}

We will need a well known result from class field theory. 
Assume that $K \subset L$ are CM-fields; then 
$\ker(N_{L/K}:\Cl(L) \to \Cl(K))$ has order $(L \cap K^1 : K)$. 
Since $K/K^+$ is ramified at the infinite places, the norm 
$N_{K/K^+}: \Cl(K) \to \Cl(K^+)$ is onto. 

\section{Hasse's unit index} 

Hasse's book \cite{H} contains numerous theorems (S\"atze 
14 -- 29) concerning the unit index $Q(L) = (E_L:W_LE_K)$, 
where $K=L^+$ is the maximal real subfield of a cyclotomic 
number field $L$. Hasse considered only abelian number fields 
$L/\Q$, hence he was able to describe these fields in terms
of their character groups $X(L)$; as we are interested in results 
on general CM-fields, we have to proceed in a different manner. 
But first we will collect some of the most elementary properties of 
$Q(L)$ (see also \cite{H} and \cite{W}; a reference "Satz $*$" 
always refers to Hasse's book \cite{H}) in

\begin{prop}\label{P1}
Let $K \subset L$ be CM-fields; then
\begin{enumerate}
\item[a)] (Satz 14) $Q(L) = (E_L:W_LE_{L^+}) = 
	(E_L^{\sigma-1}:W_L^2) = (E_L^{\sigma+1} : E_{L^+}^2)$; 
	in particular, $Q(L) \in \{1,\,2\}$.
\item[b)] (Satz 16, 17) if $Q(L) = 2$ then  $\kappa_{L/L^+} = 1$;
\item[c)] (Satz 25) If $L^+$ contains units with any given signature,
	then $Q(L) = 1$. 
\item[d)] (Satz 29) $Q(K)|Q(L)\cdot (W_L:W^{}_K)$; 
\item[e)] (compare Satz 26) Suppose that 
        $N_{L/K}: \, W_L/W_L^2 \to W^{}_K/W_K^2$ is onto. 
	Then $Q(L) \mid Q(K)$.
\item[f)] (\cite[Lemma 2]{HY}) If $(L:K)$ is odd, then $Q(L) = Q(K)$;
\item[g)] (Satz 27) If $L = \Q(\zeta_m)$, where 
	$m \not\equiv 2 \bmod 4$ is composite, then $Q(L) = 2$;
\item[h)] (see Example 4 below) Let $K_1 \subseteq \Q(\zeta_m)$ and 
   $K_2 \subseteq \Q(\zeta_n)$ be abelian CM-fields, where $m=p^\mu$ 
   and $n=q^\nu$ are prime powers such that $p \ne q$, and let
   $K = K_1K_2$; then $Q(K) = 2$. 
\end{enumerate}
\end{prop}

The proofs are straight forward:
\begin{enumerate}
\item[a)] The map $\eps \too \eps^{\sigma-1}$ induces an 
  epimorphism $E_L \longrightarrow E_L^{\sigma-1}/W_L^2$.
  If $\eps^{\sigma-1} = \zeta^2$ for some $\zeta \in W_L$, then
  $(\zeta\eps)^{\sigma-1} = 1$, and $\zeta\eps \in E_{L^+}$.
  This shows that $\sigma-1$ gives rise to an isomorphism
  $E_L/W_LE_{L^+} \longrightarrow E_L^{\sigma-1}/W_L^2,$
  hence we have $(E_L:W_LE_{L^+}) = (E_L^{\sigma-1}:W_L^2)$. 
  The other claim is proved similarly.
\item[b)] Since $W_L/W_L^2$ is cyclic of order $2$, the first claim
  follows immediately from a). Now let $\mathfrak a$ be an ideal in $\OO_K$ 
  such that $\mathfrak a\OO_L = \alpha\OO_L$. Then $\alpha^{\sigma-1} = 
  \zeta$ for some root of unity $\zeta \in L$, and  $Q(L)=2$ shows that
  $\zeta = \varepsilon^{\sigma-1}$ for some $\varepsilon \in E_L$.
  Now $\alpha\varepsilon^{-1}$ generates $\mathfrak a$ and is 
  fixed by $\sigma$, hence lies in $K$. This shows that $\mathfrak a$
  is principal in $K$, i.e. that $\kappa_{L/L^+} = 1$.
\item[c)] Units in $L^+$ that are norms from $L$ are totally
  positive; our assumption implies that totally positive units are
  squares, hence we get $E_L^{\sigma+1} = E_{L^+}^2$, and our
  claim follows from a).
\item[d)] First note that $(W_L:W^{}_K)=(W_L^2:W_K^2)$; then 
  \begin{eqnarray*} 
  Q(L)\cdot (W_L^{}:W^{}_K) & = & (E_L^{\sigma-1}:W_L^2)(W_L^2:W_K^2) \\
     & = & (E_L^{\sigma-1}:E_K^{\sigma-1})(E_K^{\sigma-1}:W_K^2) \\
     & = & (E_L^{\sigma-1}:E_K^{\sigma-1}) \cdot Q(K) \end{eqnarray*}
  proves the claim.
\item[e)] Since $Q(L) = 2$, there is a unit $\varepsilon \in E_L$ such 
  that $\varepsilon^{\sigma-1} = \zeta$ generates $W_L/W_L^2$. Taking 
  the norm to $K$ shows that $(N_{L/K}\varepsilon)^{\sigma-1} = 
  N_{L/K}(\zeta)$ generates $W^{}_K/W_K^2$, i.e. we have $Q(K)= 2$.
\item[f)] If $(L:K)$ is odd, then $(W_L:W_K)$ is odd, too, and we get
  $Q(K) \mid Q(L)$ from d) and $Q(L) \mid Q(K)$ from e). 
\item[g)] In this case, $1-\zeta_m$ is a unit, and we find
  $(1-\zeta_m)^{1-\sigma} = -\zeta_m$.  Since $-\zeta_m \in W_L^{} 
  \setminus W_L^2$, we must have $Q(L) = 2$;
\item[h)] First assume that $m$ and $n$ are odd. A 
  subfield $F \subseteq L=\Q(\zeta_m)$, where $m=p^\mu$ is an odd 
  prime power, is a CM-field if and only if it contains the maximal 
  $2$-extension contained in $L$, i.e. if and only $(L:F)$ is odd. 
  Since $(\Q(\zeta_m):K_1)$ and $(\Q(\zeta_n):K_2)$ are both odd, so 
  is $(\Q(\zeta_{mn}):K_1K_2)$; moreover, $\Q(\zeta_{mn})$ has unit 
  index $Q = 2$, hence the assertion follows from f) and g).

  Now assume that $p=2$. If $\sqrt{-1} \in K_1$, then we must have 
  $K_1 = \Q(\zeta_m)$ for $m=2^\alpha$ and some $\alpha \ge 2$ 
  (complex subfields of the field of $2^\mu$th roots of unity 
  containing $\sqrt{-1}$ necessarily have this form). Now $n$ is 
  odd and $K_2 \subseteq \Q(\zeta_n)$ is complex, hence 
  $(\Q(\zeta_n):K_2)$ is odd. By f) it suffices to show that 
  $K_1(\zeta_n) = \Q(\zeta_{mn})$ has unit index $2$, and this 
  follows from g). 

  If $\sqrt{-1} \not\in K_1$, let $\widetilde{K}_1 = K_1(i)$; 
  then $\widetilde{K}_1 = \Q(\zeta_m)$ for $m=2^\alpha$ and some 
  $\alpha \ge 2$, and in the last paragraph we have seen that 
  $Q(\widetilde{K}_1K_2) = 2$. Hence we only need to show that the 
  norm map 
  $$ N : W_{\widetilde{K}_1}/W_{\widetilde{K}_1}^2 
  				\to W^{}_{K_1}/W_{K_1}^2 $$
  is onto: since $(W_{\widetilde{K}_1K_2}:W_{\widetilde{K}_1})$ is odd,
  this implies $2 = Q(\widetilde{K}_1K_2) \mid Q(K_1K_2)$ by e). But
  the observation that the non-trivial automorphism of 
  $\Q(\zeta_m)/K_1$ maps $\zeta_m$ to $-\zeta_m^{-1}$ implies 
  at once that $N(\zeta_m) = -1$, and $-1$ generates 
  $W^{}_{K_1}/W_{K_1}^2$.
 \end{enumerate}

Now let $L$ be a CM-field with maximal real subfield $K$; we will call
$L/K$ {\it essentially ramified} if $L=K(\sqrt \alpha\,)$ and there is
a prime ideal $\fp$ in $\OO_K$ such that the exact power of $\fp$
dividing $\alpha$ is odd; it is easily seen that this does not depend
on which $\alpha$ we choose.  Moreover, every ramified prime ideal
$\fp$ above an odd prime $p$ is necessarily essentially ramified. We
leave it as an exercise to the reader to verify that our definition of
essential ramification coincides with Hasse's \cite[Sect. 22]{H}; the
key observation is the ideal equation $(4\alpha) = \fa^2 \mathfrak d$,
where $\mathfrak d = \disc (K(\sqrt \alpha\,)/K)$ and $\fa$ is an
integral ideal in $\OO_K$.

   We will also need certain totally real elements of norm $2$ in 
the field of $2^m$th roots of unity: to this end we define 
\begin{eqnarray*}
 \pi_2 & = & 2 = 2+\zeta_4+\zeta_4^{-1}, \\
 \pi_3 & = & 2+\sqrt2 = 2+\zeta_8+\zeta_8^{-1}, \\
       & & \qquad \ldots,\\
 \pi_n & = & 2+\sqrt{\pi_{n-1}} = 2+\zeta_{2^n}+\zeta_{2^n}^{-1}.
\end{eqnarray*}

Let $m \ge 2,\ L = \Q(\zeta_{2^{m+1}})$ and $K = \Q(\pi_m)$; then 
$L/K$ is an extension of type $(2,2)$ with subfields
$K_1 = \Q(\zeta_{2^m}), \, K_2 = \Q(\sqrt{\pi_m})$ and 
$K_3 = \Q(\sqrt{-\pi_m})$. Moreover, $K_2/K$ and $K_3/K$ are
essentially ramified, whereas $K_1/K$ is not.

\begin{thm}\label{T1}
Let $L$ be a CM-field with maximal real subfield $K$; 
\begin{enumerate}
\item[(i)] If $w_L \equiv 2 \bmod 4$, then 
\begin{enumerate}
\item[1.] If $L/K$ is essentially ramified, then $Q(L) = 1$, and 
    $\kappa_{L/K} = 1$. 
\item[2.] $L/K$ is not essentially ramified. Then 
    $L=K(\sqrt \alpha\,)$ for some $\alpha \in \OO_K$ such that 
    $\alpha\OO_K = \mathfrak a^2$, where $\mathfrak a$ is an integral ideal 
    in $\OO_K$. Now \newline
   ${}\ $ {\em (a)} $Q(L) = 2$, if $\mathfrak a$ is principal, and \newline
   ${}\ $ {\em (b)} $Q(L) = 1$ and 
          $\kappa_{L/K} = \langle [\mathfrak a]\rangle$,
          if $\mathfrak a$ is not principal.
\end{enumerate}
\item[(ii)] If $w_L \equiv 2^m \bmod 2^{m+1}$, where $m \ge 2$
     then $L/K$ is not essentially ramified, and 
\begin{enumerate}
\item[1.] if $\pi_m\OO_K$ is not an ideal square, then 
            $Q(L) = 1$ and $\kappa_{L/K} = 1$; 
\item[2.] if $\pi_m\OO_K = \mathfrak b^2$ for some integral 
          ideal $\mathfrak b$, then \newline
     ${}\ $ {\em (a)} $Q(L) = 2$, if $\mathfrak b$ is principal, and \newline
     ${}\ $ {\em (b)} $Q(L) = 1$, 
            $\kappa_{L/K} = \langle [\mathfrak b]\rangle$,  
             if $\mathfrak b$ is not principal. 
\end{enumerate}\end{enumerate}
\end{thm}

For the proof of Theorem \ref{T1} we will need the following

\begin{lem}
Let $L=K(\sqrt\pi\,)$, and let $\sigma$ denote the non-trivial
automorphism of $L/K$. Moreover, let $\mathfrak b$ be an ideal in 
$\OO_K$ such that $\mathfrak b\OO_L = (\beta)$ and $\beta^{\sigma-1}=-1$ 
for some $\beta \in L$. Then $\pi\OO_K$ is an ideal square in $\OO_K$.

If, on the other hand, $\beta^{\sigma-1}=\zeta$, where $\zeta$ is a
primitive $2^m$th root of unity, then $\pi_m\OO_K$ is 
an ideal square in $\OO_K$.
\end{lem}

\begin{proof}
We have $(\beta\sqrt\pi\,)^{\sigma-1}=1$, hence $\beta\sqrt\pi \in K$.
Therefore $\mathfrak b$ and $\mathfrak c = (\beta\sqrt\pi\,)$ are ideals in 
$\OO_K$, and $(\mathfrak c {\mathfrak b}^{-1})^2 = \pi\OO_K$ proves our claim. 

Now assume that $\beta^{\sigma-1}=\zeta$; then $\sigma$ fixes
$(1-\zeta)\beta^{-1}$, hence $\left((1-\zeta)\beta\right)$ and 
$\mathfrak c = (1-\zeta) = {\mathfrak c}^\sigma$ are ideals in $\OO_K$, 
and $\mathfrak c^2 = N_{L/K}(1-\zeta) = (2+\zeta+\zeta^{-1})\OO_K$ 
is indeed an ideal square in $\OO_K$ as claimed.
\end{proof}

\begin{proof}[Proof of Theorem \ref{T1}]
There are the following cases to consider:\\

\medskip \noindent 
(i) Assume that $w_L \equiv 2 \bmod 4$. 
\begin{enumerate}
  \item[1.] $L/K$ is essentially ramified. \newline
   Assume we had $Q(L)=2$; then $E_L^{\sigma-1} = W_L$, hence there 
   is a unit $\varepsilon \in E_L$ such that $\varepsilon^{\sigma-1} 
   = -1$. Write $L=K(\sqrt\pi\,)$, and apply Lemma 1 to $\mathfrak b=(1), 
   \beta = \varepsilon$: this will yield the contradiction that 
   $L/K$ is not essentially ramified.
 \item[2.] $L/K$ is not essentially ramified. \newline
   Then $L = K(\sqrt \alpha\,)$ for some $\alpha \in \OO_K$ such that 
   $\alpha\OO_K = \mathfrak a^2$, where $\mathfrak a$ is an integral ideal 
   in $\OO_K$. 
\begin{enumerate}
 \item[(a)] If $\mathfrak a$ is principal, say $\mathfrak a = \beta\OO_K$, 
   then there is a unit $\varepsilon \in E_K$ such that 
   $\alpha = \beta^2 \varepsilon$, and we see that 
   $L=K(\sqrt \varepsilon\,)$. Now
   $\sqrt \varepsilon^{\,\sigma-1}=-1$ is no square since 
   $w_L \equiv 2 \bmod 4$, and Prop. 1.a) gives $Q(L)=2$.
 \item[(b)] If $\mathfrak a$ is not principal, then the ideal 
   class $[\mathfrak a]$ capitulates in $L/K$ because $\mathfrak a 
   \OO_L = \sqrt \alpha \OO_L$. Prop. 1.b) shows that $Q(L) = 1$.
\end{enumerate}
\end{enumerate}

\medskip\noindent
(ii) Assume that $w^{}_L \equiv 2^m \bmod 2^{m+1}$ for some $m \ge 2$.
\begin{enumerate}
 \item[1.] If we have $Q(L)=2$ or $\kappa_{L/K} \ne 1$, then 
   Lemma 1 says that $\pi_m\OO_K = \mathfrak b^2$ is an ideal square in 
   $\OO_K$ in contradiction to our assumption.
 \item[2.] Suppose therefore that $\pi_m = \mathfrak b^2$ is an ideal 
   square in $\OO_K$. If $\mathfrak b$ is not principal, then 
   $\mathfrak b\OO_L = (1-\zeta)$ shows that 
   $\kappa_{L/K} = \langle [\mathfrak b]\rangle$, and Prop. 1.b)
   gives $Q(L) = 1$. If, on the other hand, $\mathfrak b = \beta\OO_K$, 
   then
   $\eta\beta^2 = \pi_m$ for some unit $\eta \in E_K$. If $\eta$ were
   a square in $\OO_K$, then $\pi_m$ would also be a square, and
   $L=K(\sqrt{-1}\,)$ would contain the $2^{m+1}$th roots of unity.
   Now $\eta\beta^2 = \pi_m = \zeta^{-1}(1+\zeta)^2$, hence 
   $\eta\zeta$ is a square in $L$, and we have $Q(L)=2$ as claimed. 
\end{enumerate}
\end{proof}

\medskip\noindent{\bf Remark.} 
For $L/\Q$ abelian, Theorem \ref{T1} is equivalent
to Hasse's Satz 22; we will again only sketch the proof: suppose 
that $w_L \equiv 2^m \bmod 2^{m+1}$ for some $m \ge 2$, and define 
$L' = L(\zeta_{2^{m+1}}), \, K' = L' \cap \R$. Then
$K'/K$ is essentially ramified if and only if $\pi_{m}$ is not 
an ideal square in $\OO_K$ (because $K' = K(\pi_{m+1}) = 
K(\sqrt{\pi_m})$). The asserted equivalence should now be clear.
Except for the results on capitulation, Theor. \ref{T1} is also contained 
in \cite{O} (for general CM-fields). 

\medskip
\noindent{\bf Examples.}

1. Complex subfields $L$ of $\Q(\zeta_{p^m})$, where $p$ is 
   prime, have unit index ${Q(L)=1}$ (Hasse's Satz 23) and 
   $\kappa_{L/L^+} = 1$: since $p$ ramifies completely in 
   $\Q(\zeta_{p^m})/\Q$, $L/L^+$ is essentially ramified if 
   $p \ne 2$, and the claim follows from Theorem 1. 
   If $p=2$ and $L/L^+$ is not essentially ramified, then we must 
   have $L = \Q(\zeta_{2^\mu})$ for some $\mu \in \N$, and we
   find $Q(L) = 1$ by Theorem 1.2.1.

2. $L = \Q(\zeta_m)$ has unit index $Q(L) = 1$ if and only
   if $m \not\equiv 2 \bmod 4$ is a prime power (Satz 27). This 
   follows from Example 1. and Prop. \ref{P1}.e)

3. If $K$ is a CM-field, which is essentially ramified at a
   prime ideal $\fp$ above $p \in \N$, and if $F$ is a
   totally real field such that $p \nmid \disc F$, then $Q(L) = 1$
   and $\kappa^{}_{L/L^+}=1$ for $L=KF$: this is again due to the 
   fact that either $L/L^+$ is essentially ramified at the prime 
   ideals above $\fp$, or $p = 2$ and $K = K^+(\sqrt{-1}\,)$. In 
   the first case, we have $Q(L) = 1$ by Theorem 1.1.1, and in the 
   second case by Theorem 1.2.1.  

4. Suppose that the abelian CM-field $K$ 
   is the compositum $K = K_1\ldots K_t$ of fields with pairwise 
   different prime power conductors; then $Q(K) = 1$ if and only if 
   exactly one of the $K_i$ is imaginary. (Uchida \cite[Prop. 3]{U})
   The proof is easy: if there is exactly one complex field among the
   $K_j$, then $Q(K) = 1$ by Example 3. Now suppose that $K_1$ and $K_2$
   are imaginary; we know $Q(K_1K_2) = 2$ (Prop. 1.h), and from the 
   fact that the $K_j$ have pairwise different conductors we deduce
   that $(W_K:W_{K_1K_2}) \equiv 1 \bmod 2$. Now the claim follows 
   from Hasse's Satz 29 (Prop. 1.c). Observe that $\kappa_{K/K^+} = 1$
   in all cases.

5. Cyclic extensions $L/\Q$ have unit index $Q(L) = 1$ (Hasse's
   Satz 24): Let $F$ be the maximal subfield of $L$ such that $(F:\Q)$ 
   is odd. Then $F$ is totally real, and $2 \nmid \disc F$ (this 
   follows from the theorem of Kronecker and Weber). Similarly, let 
   $K$ be the maximal subfield of $L$ such that $(K:\Q)$ is a 
   $2$-power: then $K$ is a CM-field, and $L = FK$. If $K/K^+$ is 
   essentially ramified at a prime ideal $\fp$ above an odd prime $p$, 
   then so is $L/L^+$, because $L/\Q$ is abelian, and all prime ideals 
   in $F$ have odd ramification index. Hence the claim in this case 
   follows by Example 3. above.

   If, however, $K/K^+$ is not essentially ramified at a prime ideal 
   $\fp$ above an odd prime $p$, then $\disc K$ is a $2$-power (recall 
   that $K/\Q$ is cyclic of $2$-power degree). Applying the theorem 
   of Kronecker and Weber, we find that $K \subseteq \Q(\zeta)$,
   where $\zeta$ is some primitive $2^m$th root of unity.
   If $K/K^+$ is essentially ramified at a prime ideal above $2$,
   then so is $L/L^+$, and Theorem 1 gives us $Q(L)=1$.
   If $K/K^+$ is not essentially ramified at a prime ideal above $2$,
   then we must have $K=\Q(\zeta)$, where $\zeta$ is a primitive 
   $2^m$th root of unity; but now $\pi_{m}\OO_{L^+}$ is not
   the square of an integral ideal, and we have $Q(L)=1$ by
   Theorem 1. Alternatively, we may apply Prop. 1.e) and observe
   that $Q(K) = 1$ by Example 1.

6. Let $p \equiv 1 \bmod 8$ be a prime such that the 
   fundamental unit $\eps_{2p}$ of $\Q(\sqrt{2p}\,)$ has norm $+1$
   (by \cite{S}, there are infinitely many such primes; note also 
   that $N\eps_{2p} = +1 \iff (2,\sqrt{2p}\,)$ is principal).
   Put $K = \Q(i,\sqrt{2p}\,)$ and $L = \Q(i,\sqrt{2},\sqrt{p}\,)$.
   Then $Q(K) = 2$ by Theor. 1.2.2.a), whereas the fact that $L$
   is the compositum of $\Q(\zeta_8)$ and $\Q(\sqrt{p}\,)$ shows
   that $Q(L) = 1$ (Example 4). This generalization of Lenstra's 
   example given by Martinet in \cite{H} is contained in Theor. 4
   of \cite{HY}, where several other results of this kind can be 
   found. 

\section{Masley's theorem $h^-_m|h^-_{mn}$} 

Now we can prove a theorem that will contain Masley's result 
$h^-(K)\,|\,h^-(L)$ for cyclotomic fields $K = \Q(\zeta_m)$ and 
$L = \Q(\zeta_{mn})$ as a special case: 

\begin{thm}\label{T2}
Let  $K \subset L$ be CM-fields; then  
$$h^-(K)\,|\,h^-(L)\cdot |\kappa_{L/L^+}\,|
   \cdot\frac{(L\cap K^1:K)}{(L^+ \cap (K^+)^1:K^+)},$$
and the last quotient is a power of $2$.
\end{thm}

\begin{proof}
Let $\nu_K$ and $\nu_L$ denote the norms $N_{K/K^+}$ and $N_{L/L^+}$, 
respectively; then the 
following diagram is exact and commutative:

$$\begin{CD}
 1 @>>> \Cl^-(L) @>>> \Cl(L) @>\nu_L>> \Cl(L^+) @>>> 1 \\ 
 @.   @VVN^-V   @VVNV     @VVN^+V               @. \\
 1 @>>> \Cl^-(K) @>>> \Cl(K) @>\nu_K>> \Cl(K^+) @>>> 1 
\end{CD}$$

The snake lemma gives us an exact sequence
$$ \begin{CD}
  1 @>>> \ker N^- @>>>  \ker N @>>> \ker N^+ \\
    @>>> \cok N^- @>>>  \cok N  @>>>  \cok N^+ @>>> 1. \end{CD} $$
Let $h(L/K)$ denote the order of $\ker N$, and let $h^-(L/K)$
and $h(L^+/K^+)$ be defined accordingly. 
The remark at the end of Sect. 1 shows
$$ |\,\cok N  \,| = (L\cap K^1:K), \quad 
   |\,\cok N^+\,| = (L^+\cap (K^+)^1:K^+).$$ 
The alternating product of the orders of the groups in 
exact sequences equals $1$, so the above sequence implies
$$ h^-(L/K) \cdot h(L^+/K^+) \cdot |\cok N| =
   h(L/K) \cdot |\cok N^-| \cdot |\cok N^+|.$$
The exact sequence
$$ 1 \to \ker N^- \to \Cl^-(L) \to \Cl^-(K) \to \cok N^- \to 1$$
gives us 
$$ h^-(L/K) \cdot h^-(K) = h^-(L) \cdot |\cok N^-|.$$
Collecting everything we find that
\begin{equation}\label{E1}
h^-(K) \cdot \frac{h(L/K)}{h(L^+/K^+)} \cdot 
\frac{(L^+\cap (K^+)^1:K^+)}{(L\cap K^1:K)} = h^-(L) 
\end{equation}
Now the claimed divisibility property follows if we can prove
that $h(L^+/K^+)$ divides $h(L/K)\cdot|\,\kappa_{L/L^+}\,|$.
But this is easy: exactly $h(L^+/K^+)/|\,\kappa_{L/L^+}\,|$ 
ideal classes of $\ker N^+ \subset \Cl(L^+)$ survive the transfer 
to $\Cl(L)$, and if the norm of $L^+/K^+$ kills an ideal class 
$c \in \Cl(L^+)$, the same thing happens to the transferred class 
$c^j$ when the norm of $L/K$ is applied. We remark in passing that 
$|\,\kappa_{L/L^+}\,| \le 2$ (see Hasse \cite[Satz 18]{H}).
    
It remains to show that $(L\cap K^1:K)/(L^+ \cap (K^+)^1:K^+)$  
is a power of $2$. Using induction on $(L:K)$, we see that it
suffices to prove: if $L/K$ is an unramified abelian extension of 
CM-fields of odd prime degree $(L:K) = q$, then so is $L^+/K^+$.
Suppose otherwise; then there exists a finite prime $\fp$ which ramifies, 
and since $L^+/K^+$ is cyclic, $\fp$ has ramification index $q$. 
Now $L/K^+$ is cyclic of order $2q$, hence $K$ must be the inertia field 
of $\fp$, contradicting the assumption that $L/K$ be unramified. We 
conclude that $L^+/K^+$ is also unramified, and so odd factors of 
$(L\cap K^1:K)$ cancel against the corresponding factors of 
$(L^+\cap (K^+)^1:K^+)$. 
\end{proof}

\begin{cor}[Louboutin and Okazaki \cite{LO}]\label{C1}
Let $K \subset L$ be CM-fields such that $(L:K)$ is odd; then 
$h^-{K} \mid h^-(L)$.
\end{cor}

\begin{proof} 
From (\ref{E1}) and the fact that $(L^+\cap (K^+)^1:K^+) = 1$ (this
index is a power of $2$ and divides $(L:K)$, which is odd), we 
see that it is sufficent to show that $h(L^+/K^+)\mid h(L/K)$.
This in turn follows if we can prove that no ideal class from
$\ker N^+ \subseteq \Cl(L^+)$ capitulates when transferred to 
$\Cl(L)$. 
Assume therefore that $\kappa_{L/L^+} = \langle [{\mathfrak a}] \rangle$.
If $w_L \equiv 2 \bmod 4$, then by  Theor. 1.1.2 we may assume
that $L = L^+(\sqrt{\alpha}\,)$, where $\alpha\OO_{L^+} = {\mathfrak a}^2$.
Since $(L:K)$ is odd, we can choose $\alpha \in O_{K^+}$, hence
 $N^{+}({\mathfrak a}) = {\mathfrak a}^{(L:K)}$ shows that the ideal class
$[{\mathfrak a}]$ is not contained in $\ker N^+$. The proof in the case 
$w_L \equiv 0 \bmod 4$ is completely analogous.  
\end{proof} 

\medskip\noindent{\bf Remark.} 
For any prime $p$, let $\Cl^-_p(K)$ denote the 
$p$-Sylow subgroup of $\Cl^-(K)$; then $\Cl^-_p(K) \subseteq 
\Cl^-_p(L)$ for every $p \nmid (L:K)$. This is trivial, because
ideal classes with order prime to $(L:K)$ cannot capitulate in $L/K$. 
\medskip

\begin{cor}[Masley \cite{MM}]\label{C2}
If $K = \Q(\zeta_m)$ and $L = \Q(\zeta_{mn})$ for some $m,n \in \N$,
then $h^-(K)\,|\,h^-(L).$
 \end{cor}

\begin{proof} 
We have shown in Sect.$\,$2 that $j_{K^+ \to K} \text{ and } 
j_{L^+ \to L} $ are injective in this case. Moreover, $L/K$ does not
contain a nontrivial subfield of $K^1$ (note that $p$ is completely
ramified in $L/K$ if $n=p$, and use induction).
\end{proof} 

The special case $m=p^a,\,n=p$ of Corollary 2 can already be found 
in \cite{We}. Examples of CM-fields $L/K$ such that $h^-(K) \nmid 
h^-(L)$ have been given by Hasse \cite{H}; here are some more:
\begin{enumerate}
\item Let $d_1 \in \{-4, -8, -q\ (q \equiv 3 \bmod 4)\}$ 
   be a prime discriminant, and suppose that $d_2 > 0$ is the 
   discriminant of a real quadratic number field such that 
   $(d_1,d_2) = 1$. Put $K = \Q(\sqrt{d_1d_2}\,)$ and 
   $L=\Q(\sqrt{d_1},\sqrt{d_2}\,)$; then $Q(L) = 1$ and 
   $\kappa_{L/L^+} = 1$ by Example 4, and $(L\cap K^1:K) = 
   2\cdot (L^+ \cap (K^+)^1:K^+)$ since $L/K$ is unramified but
   $L^+/K^+$ is not. The class number formula (1) below shows that
   in fact $h^-(K) \nmid h^-(L)$.
\item Let $d_1=-4, d_2 = 8m$ for some odd $m \in \N$, and suppose 
   that $\mathfrak 2 = (2,\sqrt{2m}\,)$ is not principal in $\OO_k$, where
   $k = \Q(\sqrt{2m}\,)$. Then $h^-(K) \nmid h^-(L)$ for 
   $K = \Q(\sqrt{-2m}\,),\, L = \Q(\sqrt{-1}, \sqrt{2m}\,)$.
   Here $(L\cap K^1:K) = (L^+ \cap (K^+)^1:K^+),$ but $\kappa_{L/L^+} 
   = \langle [\mathfrak 2]\rangle$, since $\mathfrak 2 \OO_L = (1+i)$.
   This example shows that we cannot drop the factor 
   $\kappa_{L/L^+}$ in Theor. 2.
\end{enumerate}
Other examples can be found by replacing $d_1$ in Example 2. by 
$d_1 = -8$ or $d_2$ by $d_2 = 4m, m\in \N$ odd. The proof that 
in fact $h^-(K) \nmid h^-(L)$ for these fields uses Theorem \ref{T1}, 
as well as Prop. \ref{P2} and \ref{P3} below.

\section{Mets\"{a}nkyl\"{a}'s factorization}

An extension $L/K$ is called a $V_4$-extension of CM-fields if
\begin{enumerate}
\item $L/K$ is normal and $\Gal(L/K) \simeq V_4 = (2,2)$;
\item Exactly two of the three quadratic subfields are CM-fields;
         call them $K_1$ and $K_2$, respectively.
\end{enumerate}
This implies, in particular, that $K$ is  totally real, and that $L$
is a CM-field with maximal real subfield $L^+= K_3$. We will write
$Q_1 = Q(K_1), W_1 = W_{K_1}, $ etc. 

Louboutin \cite[Prop. 13]{Lou} has given an analytic proof of the
following class number formula for $V_4$-extension of CM-fields, 
which contains Lemma 8 of Ferrero \cite{F} as a special case:

\begin{prop}\label{P2}
Let $L/K$ be a $V_4$-extension of CM-fields; then
$$ h^-(L) = \frac{Q(L)}{Q_1Q_2} \frac{w^{}_L}{w^{}_1w^{}_2} 
            h^-(K_1)h^-(K_2). $$
\end{prop}

\begin{proof} Kuroda's class number formula (for an algebraic proof 
see \cite{L}) yields 
\begin{equation}\label{E2}
 h(L) = 2^{d-\kappa-2-\upsilon}q(L)h(K_1)h(K_2)h(L^+)/h(K)^2, 
\end{equation}
where 
\begin{itemize}
\item $d = (K:\Q)$ is the number of infinite primes of $K$ that 
           ramify in $L/K$;
\item $\kappa = d-1$ is the $\Z$-rank of the unit group of $K$;
\item $\upsilon= 1$ if and only if all three quadratic subfields
          of $L/K$ can be written as $K(\sqrt\varepsilon\,)$ for 
          units $\varepsilon \in E_K$, and $\upsilon = 0$ otherwise;
\item $q(L)= (E_L:E_1E_2E_3)$ is the unit index for extensions
          of type $(2,2)$; here $E_j$ is the unit group of $K_j$ 
          (similarly, let $W_j$ denote the group of roots unity in
           $L_j$).
\end{itemize}

Now we need to find a relation between the unit indices involved;
we claim

\begin{prop}\label{P3}
If $L/K$ is a $V_4$-extension of CM-fields, then
$$ \frac{Q(L)}{Q_1Q_2} \frac{w_L}{w_1w_2} = 2^{-1-\upsilon}q(L). $$
\end{prop}
 
Plugging this formula into equation (\ref{E2}) and recalling 
$h^-(F) = h(F)/h(F^+)$ for CM-extensions $F/F^+$ yields Louboutin's 
formula.
\end{proof}

\begin{proof}[Proof of Prop. \ref{P3}]
We start with the observation
$$ Q(L) = (E_L:W_LE_3) = (E_L:E_1E_2E_3) 
         \frac{(E_1E_2E_3:W_1W_2E_3)}{(W_LE_3:W_1W_2E_3)}. $$
In \cite{L} we have defined groups $E_j^* = \{\eps \in E_j: 
N_j \eps \text{ is a square in }E_K\}$, where $N_j$ denotes the 
norm of $K_j/K$; we also have shown that
$$ (E_1E_2E_3:E_1^*E_2^*E_3^*) = 2^{-\upsilon} \prod (E_j:E_j^*) $$
and $E_j/E_j^* \simeq E_K/N_jE_j$. Now Prop. \ref{P1}.a) gives 
$(E_K:N_jE_j) = Q_j$ for $j=1, 2$, and we claim

\begin{enumerate}
\item $(W_LE_3:W_1W_2E_3) = (W_L:W_1W_2) =
             2 \cdot \frac{w^{}_L}{w^{}_1w^{}_2};$
\item $ E_1^*E_2^*E_3^* = W_1W_2E_3^*; $  
\item $(W_1W_2E_3:W_1W_2E_3^*) = (E_3:E_3^*).$
\end{enumerate}
This will give us
\begin{equation}\label{E3}
 Q(L) = 2^{-1-\upsilon}q(L)Q_1Q_2 \frac{w^{}_1w^{}_2}{w^{}_L}, 
\end{equation}  
completing the proof of Prop. \ref{P3} except for the three claims above:
\begin{enumerate}
\item $ W_LE_3/W_1W_2E_3 \simeq W_L/(W_L \cap W_1W_2E_3) 
                             \simeq W_L/W_1W_2 $,
     and the claim follows from $W_1 \cap W_2 = \{-1, +1\}$;
\item We only need to show that $E_1^*E_2^*E_3^* \subset W_1W_2E_3^*$;
but Prop. \ref{P1}.a) shows that 
$\eps \in E_1^* \iff \eps^{\sigma+1} \in E_K^2 \iff \eps \in W_1E_K$,
and this implies the claim;      
\item $ W_1W_2E_3/W_1W_2E_3^* \simeq E_3/E_3 \cap W_1W_2E_3^*
                \simeq E_3/E_3^* $.    
\end{enumerate} 
\end{proof}

Combining the result of Section 3 with Prop. \ref{P2}, 
we get the following

\begin{thm}\label{T3}
Let $L_1$ and $L_2$ be CM-fields, and let $L=L_1L_2$ and 
$K =L_1^+L_2^+$; then $L/K$ is a $V_4-$extension of CM-fields with 
  subfields $K_1 = L_1L_2^+,\ K_2=L_1^+L_2, \ K_3=L^+$, and
  $$ h^-(L) = \frac{Q(L)}{Q_1Q_2} \frac{w^{}_L}{w^{}_1w^{}_2} 
     h^-(L_1)h^-(L_2)T_1T_2,$$
where $T_1 = h^-(L_1L_2^+)/h^-(L_1)$ and 
$T_2 = h^-(L_2L_1^+)/h^-(L_2).$ 
\newline
If we assume that $\kappa_1 = \kappa_2 = 1$ ($\kappa_1$ 
is the group of ideal classes capitulating in $L_1L_2^+/K,\,\kappa_2$ 
is defined similarly) and that 
\begin{eqnarray*}
(L_1L_2^+ \cap L_1^1 : L_1) & = (L_1^+L_2^+ \cap (L_1^+)^1 : L_1^+), \\
(L_2L_1^+ \cap L_2^1 : L_2) & = (L_2^+L_1^+ \cap (L_2^+)^1 : L_2^+),
\end{eqnarray*}
then $T_1$ and $T_2$ are integers. 
\end{thm}

\begin{proof} 
Theorem \ref{T3} follows directly from Theorem \ref{T2} and Proposition 
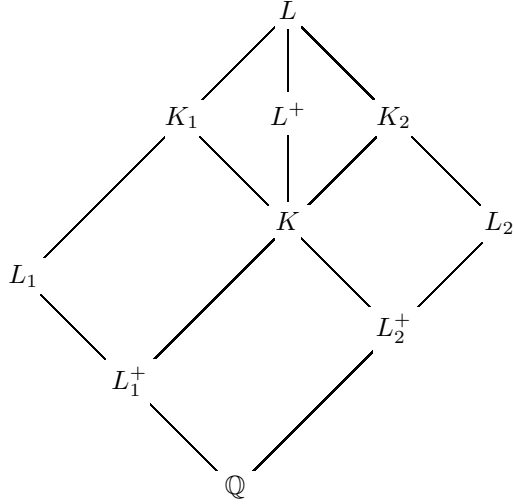
\ref{P2}. 
\end{proof}

\begin{diagram}[width=2em,height=2em]
& & &     &        &   L     &         & & & & \\
& & &     & \ruLine & \vLine & \luLine &     & & & \\
& & & K_1 &         &   L^+  &         & K_2 & & & \\ 
& &  \ruLine(3,3) &  & \rdLine & \vLine & \ldLine &  & \luLine & & \\
& & &     &        &   K     &         & & &  L_2 & \\
 L_1 & & &     &  \ldLine(3,3)    &    & \rdLine & & \ruLine & & \\
& \luLine & &  &  & & & L_2^+ \\
& & L_1^+ &  &  &  & \ruLine(3,3) &  \\
& &  & \rdLine &  & & &  \\
& &  &  & \Q & & &  \\
\end{diagram}

\bigskip


\medskip

Now let $m=p^\mu$ and $n=q^\nu$ be prime powers, and suppose that 
$p \ne q$ Moreover, let $L_1 \subseteq \Q(\zeta_m)$ and 
$L_2 \subseteq \Q(\zeta_n)$ be CM-fields. Then 
\begin{enumerate}
\item $Q(L)=2,\,Q_1 = Q(L_1L_2^+) = Q_2 = Q(L_2L_1^+)=1$: 
   this has been proved in Prop. \ref{P1}.h) and Example 4 
   in Sect. 2;
\item $w_1w_2 = 2w_L$ (obviously);
\item $\kappa_1=\kappa_2=1$: see Example 4 in Sect. 2;
\item $(L_1L_2^+\cap L_1^1:L_1)=(L_1^+L_2^+ \cap (L_1^+)^1:L_1^+)$:  
        this, as well as the corresponding property for $K_2$, is 
        obvious, because the prime ideals above $p$ and $q$ 
        ramify completely in $L/L_2$ and $L/L_1$, respectively.
\end{enumerate}

In particular, we have the following

\begin{cor} (Mets\"ankyl\"a) 
Let $L_1 \subseteq \Q(\zeta_m)$ and $L_2 \subseteq \Q(\zeta_n)$ 
be CM-fields, where $m=p^\mu$ and $n=q^\nu$ are prime powers, 
and let $L=L_1L_2$; then
$$ h^-(L) = h^-(L_1)h^-(L_2)T_1T_2,$$
where  $T_1 = h^-(L_1L_2^+)/h^-(L_1)$ and 
$T_2 = h^-(L_2L_1^+)/h^-(L_2)$ are integers.
\end{cor}

It still remains to identify the character sums $T_{01}$ 
and $T_{10}$ in \cite{M} with the class number factors $T_1$ and
$T_2$ given above. But this is easy: the character group $X(L_1)$
corresponding to the field $L_1$ is generated by a character
$\chi_1$, and it is easily seen that 
$$\begin{array}{rlcrlc} 
   X(L_1)      & = &\langle\chi_1\rangle, & 
   X(L_1L_2^+) & = &\langle\chi_1,\chi_2^2\rangle, \\
   X(L_2)      & = &\langle\chi_2\rangle, & 
   X(L_2L_1^+) & = &\langle\chi_2,\chi_1^2\rangle, \\
   X(L)        & = &\langle\chi_1,\chi_2\rangle,  & 
   X(L^+)      & = &\langle\chi_1\chi_2,\chi_1^2\rangle.
\end{array}$$

The analytical class number formula for an abelian CM-field $K$ reads
\begin{equation}\label{E4}
 h^-(K) = Q(K)w^{}_K \prod_{\chi \in X^-(K)} \frac{1}{2\mathfrak f(\chi)}
        \sum_{a\, {\bmod}^+ \,\mathfrak f(\chi)} (-\chi(a)a), 
\end{equation}
where $a\, {\bmod}^+ \,\mathfrak f(\chi)$ indicates that the sum is 
extended over all $1 \le a \le \mathfrak f(\chi)$ such that 
$(a,\mathfrak f(\chi)) = 1 $, and $X^-(L) = X(L) \setminus X(L^+)$ is the 
set of $\chi \in X(L)$ such that $\chi(-1) = -1$. Applying formula 
(\ref{E4}) to the CM-fields listed above and noting that 
$Q(L)=2$, $Q(L_1) = Q(L_2) = Q(L_1L_2^+) = Q(L_2L_1^+) = 1, 
 2w_L=w_1w_2$, we find
$$ h^-(L) = h^-(L_1)\cdot h^-(L_2) \prod_{\chi \in X^*(L)} 
            \frac{1}{2\mathfrak f(\chi)}
            \sum_{a\, {\bmod}^+ \,\mathfrak f(\chi)} (-\chi(a)a), $$
where $X^*(L)$ is the subset of all $\chi \in X^-(L)$
not lying in $X^-(L_1)$ or $X^-(L_2)$. Now define 
$X_1(L) = \{\chi = \chi_1^x\chi_2^y \in X^*(L): 
\, x \equiv 1 \bmod 2, \,y \equiv 0 \bmod 2 \}$, and let $X_2(L)$ be 
defined accordingly. Then $ X^*(L) = X_1(L) \cup X_2(L)$, and 
$$ h^-(L_1) \cdot \prod_{\chi \in X_1(L)} \frac{1}{2\mathfrak f(\chi)}
     \sum_{a\, {\bmod}^+ \,\mathfrak f(\chi)} 
     (-\chi(a)a) = h^-(L_1L_2^+),$$
and we have shown that 
$$ T_1 = \prod_{\chi \in X_1(L)} \frac{1}{2\mathfrak f(\chi)}
       \sum_{a\, {\bmod}^+ \,\mathfrak f(\chi)} (-\chi(a)a). $$
Comparing with the definition of Mets\"{a}nkyl\"{a}'s factor $T_{10}$,
this shows that indeed $T_1 = T_{10}$.  

\section*{Acknowldegement}

I would like to thank St\'ephane Louboutin and Ryotaro Okazaki
as well as Tauno Mets\"ankyl\"a for several helpful suggestions 
and for calling my attention to the papers of Horie \cite{Ho} and 
Uchida \cite{U}.


\begin{thebibliography}{99}

\bibitem{F} B. Ferrero,
{\em The cyclotomic $\Z_2-$extension of imaginary quadratic number 
       fields},  American J. Math. {\bf 102} (1980), 447--459
%

\bibitem{H} H. Hasse,
{\em \"{U}ber die Klassenzahl abelscher Zahlk\"{o}rper},
Springer Verlag 1985
%

\bibitem{Ho} K. Horie
{\em On a ratio between relative class numbers},
Math. Z. {\bf 211} (1992), 505--521
%

\bibitem{HY} M. Hirabayashi), K. Yoshino
{\em Remarks on unit indices of imaginary abelian number fields},
Manuscripta Math. {\bf 60} (1988), 423--436
%

\bibitem{L} F. Lemmermeyer,
{\em Kuroda's class number formula},  
Acta Arith. {\bf 66} (1994), 245--260
%

\bibitem{Lou} S. Louboutin,
{\em Determination of all quaternion octic CM-fields with 
       class number 2}, preprint 
%

\bibitem{LO} S. Louboutin, M. Olivier,
{\em The class number one problem for some non-abelian normal
       CM-fields},  preprint 1994
%

\bibitem{M} T. Mets\"{a}nkyl\"{a},
{\em \"{U}ber den ersten Faktor der Klassenzahl des Kreisk\"{o}rpers},
Ann. Acad. Sci. Fenn. Ser. A. I {\bf 416} (1967)
%

\bibitem{MM} J. M. Masley, H. L. Montgomery,
{\em Cyclotomic fields with unique factorization},
J. Reine Angew. Math. {\bf 286/287} (1976), 248--256
%

\bibitem{O} R. Okazaki,
{\em On evaluation of L-functions over real quadratic fields},
J. Math. Kyoto Univ. {\bf 31} (1991), 1125--1153
%

\bibitem{S} A. Scholz,
{\em \"Uber die L\"osbarkeit der Gleichung $t^2-Du^2 = -4$}, 
Math. Z. {\bf 39} (1934), 95--111
%

\bibitem{U} K. Uchida,
{\em Imaginary quadratic number fields with class number one}, 
Tohoku Math. Journ. {\bf 24} (1972), 487--499
%

\bibitem{W} L. Washington,
{\em Introduction to Cyclotomic Fields},
Springer Verlag 1982
%

\bibitem{We} J. Westlund,
{\em On the class number of the cyclotomic number field},
Trans. Amer. Math. Soc. {\bf 4} (1903), 201--212
%

\end{thebibliography}
\end{document}